\documentclass{amsart}
\usepackage{hyperref}
\usepackage{fullpage}
\usepackage{amsrefs}

\input xypic
\newdir{ >}{{}*!/-9pt/\dir{>}}

\newcommand{\Ab}	{\operatorname{Ab}}
\newcommand{\Ch}        {\operatorname{Ch}}
\newcommand{\Ext}       {\operatorname{Ext}}
\newcommand{\Hom}       {\operatorname{Hom}}

\newcommand{\cok}       {\operatorname{cok}}
\newcommand{\img}       {\operatorname{image}}

\newcommand{\al}        {\alpha}
\newcommand{\bt}        {\beta} 

\newcommand{\dl}        {\delta}
\newcommand{\ep}        {\epsilon}
\newcommand{\tht}       {\theta}

\newcommand{\Gm}        {\Gamma}
\newcommand{\Sg}        {\Sigma}

\newcommand{\CI}        {{\mathcal{I}}}
\newcommand{\CJ}        {{\mathcal{J}}}

\newcommand{\tg}        {\tilde{g}}
\newcommand{\ti}        {\tilde{\imath}}
\newcommand{\tk}        {\tilde{k}}
\newcommand{\tn}        {\tilde{n}}
\newcommand{\tp}        {\tilde{p}}
\newcommand{\tq}        {\tilde{q}}

\newcommand{\Z}         {{\mathbb{Z}}}
\newcommand{\ov}[1]     {\overline{#1}}
\newcommand{\sm}        {\setminus}
\newcommand{\xla}       {\xleftarrow}
\newcommand{\xra}       {\xrightarrow}
\newcommand{\CA}        {{\mathcal{A}}}
\newcommand{\bsm}       {\left[\begin{smallmatrix}}
\newcommand{\esm}       {\end{smallmatrix}\right]}
\newcommand{\op}        {\oplus}
\newcommand{\ot}        {\otimes}
\newcommand{\st}        {\;|\;}

\renewcommand{\:}{\colon}

\newtheorem{theorem}{Theorem}

\newtheorem{lemma}[theorem]{Lemma}
\newtheorem{proposition}[theorem]{Proposition}
\newtheorem{corollary}[theorem]{Corollary}
\theoremstyle{definition}
\newtheorem{remark}[theorem]{Remark}
\newtheorem{definition}[theorem]{Definition}

\newtheorem{construction}[theorem]{Construction}

\begin{document}
\title{The model structure for chain complexes}
\author{N.~P.~Strickland}

\maketitle 

Let $\Ch$ be the category of (possibly unbounded) chain complexes of
abelian groups.  In this note we construct the standard Quillen model
structure on $\Ch$, by a method that is somewhat different from the
standard one.  Essentially, we use a functorial two-stage projective
resolution for abelian groups, and build everything directly from
that.  This has the advantage of being very concrete, explicit and
functorial.  It does not rely on the small object argument, or make
any explicit use of transfinite induction.  On the other hand, it is
not so conceptual, and it does use the fact that subgroups of free
abelian groups are free, so it does not generalise to many rings other
than $\Z$ (but see Remark~\ref{rem-ChR}).  We do not claim any great
technical benefit for this approach, but it seems like an interesting
alternative, and may be pedagogically useful.

We refer to~\cite{ho:mc} for the general theory of model categories.

\begin{definition}\label{defn-main}
 Consider a map $f\:A_*\to B_*$ in $\Ch$.  We say
 that: 
 \begin{itemize}
  \item $f$ is a cofibration iff it is injective, and
   $\cok(f)_n$ is a free abelian group for all $n$.
  \item $f$ is a fibration iff it is surjective.
  \item $f$ is a weak equivalence iff it is a
   quasiisomorphism, which means that $f_*\:H_*A\to H_*B$ is an
   isomorphism.
  \item $f$ is an acyclic cofibration iff it is a cofibration and a
   weak equivalence, or equivalently a monomorphism whose cokernel is
   acyclic and free in each degree.
  \item $f$ is an acyclic fibration iff it is a fibration and a weak
   equivalence, or equivalently an epimorphism with acyclic kernel.
 \end{itemize}
\end{definition}

\begin{theorem}\label{thm-main}
 This defines a cofibrantly generated model structure on $\Ch$, which
 is left and right proper and also monoidal.
\end{theorem}

The rest of this document will constitute the proof.

\begin{definition}
 For any abelian group $A$, we let $I(A)$ be the kernel of the
 augmentation $\ep\:\Z[A]\to\Z$.  Let $\tht\:\Z[A]\to A$ be the usual
 map given by $\tht(\sum_in_i[a_i])=\sum_in_ia_i$, and put
 $I^2(A)=\ker(\tht\:I(A)\to A)$.  Note that $I$ and $I^2$ give
 functors $\Ab\to\Ab$ that are nonadditive, but they do preserve zero
 morphisms, so they induce functors $\Ch\to\Ch$.  Note also that
 $I(A)$ is freely generated by elements $[a]-[0]$ for $a\in A\sm 0$.
 As subgroups of free abelian groups are always free we see that
 $I^2(A)$ is also free, although in this case there is no obvious
 choice of basis.  
\end{definition}

\begin{remark}
 It is well known and not hard to check that $I^2(A)$ is the square of
 the ideal $I(A)$ in the ring $\Z[A]$, but we will not need this.
\end{remark}

\begin{construction}\label{cons-acf-fib}
 Let $f\:A_*\to B_*$ be a chain map.  Define a graded group
 $W_*=W_*(f)$ by  
 \[ W_* = A_*\op I(B_*)\op \Sg^{-1}I(B_*). \]
 We write $a,\bt$ and $\bt'$ for typical elements of the groups $A_n$,
 $I(B_n)$ and $I(B_{n+1})$.  We define $d^W\:W_n\to W_{n-1}$ by
 \[ d^W(a+\bt+\Sg^{-1}\bt') = da + d\bt - \Sg^{-1}(\bt+d\bt').
 \]
 We also let $j\:A_*\to W_*$ be the inclusion, and define
 $p\:W_*\to B_*$ by $p(a+\bt+\Sg^{-1}\bt')=f(a)+\tht(\bt)$.
\end{construction}

\begin{proposition}\label{prop-acf-fib}
 $W_*$ is a chain complex, and $j$ and $p$ are chain maps with $pj=f$.
 Moreover, $j$ is an acyclic cofibration and $p$ is a fibration.
\end{proposition}
\begin{proof}
 It is a straightforward computation to check that $d^Wd^W=0$, so
 $W_*$ is a chain complex.  It is visible that $d^Wj=jd$, and it is
 again straightforward to check that $pd^W=dp$, so $j$ and $p$ are
 chain maps, and clearly $pj=f$.  The map $j$ is injective, and the
 cokernel is just the cone on $\Sg^{-1}I(B_*)$, which is contractible
 and free in each degree; this means that $j$ is an acyclic
 cofibration.  Also, as $\tht$ is surjective we see that $p$ is an
 epimorphism and thus a fibration.
\end{proof}

\begin{construction}\label{cons-cof-afb}
 Let $f\:A_*\to B_*$ be a chain map.  Define a graded group
 $X_*=X_*(f)$ by  
 \[ X_* = A_*\op\Sg I(A_*)\op \Sg^2I^2(A_*) \op 
                   I(B_*) \op \Sg I(B_*).
 \]
 We write $a,\al',\al'',\bt$ and $\bt'$ for typical elements of the
 groups $A_n$, $I(A_{n-1})$, $I^2(A_{n-2})$, $I(B_n)$ and
 $I^2(B_{n-1})$.  We define $d^X\:X_n\to X_{n-1}$ by 
 \begin{align*}
  d^X(a) &= da & \in & A_{n-1} \\
  d^X(\Sg\al') &= \tht(\al') -\Sg d\al' - f_*(\al') 
    & \in & A_{n-1} \op \Sg I(A_{n-2}) \op I(B_{n-1}) \\
  d^X(\Sg^2\al'') &= \Sg\al'' + \Sg^2 d\al'' + \Sg f_*\al'' 
    & \in & \Sg I(A_{n-2}) \op \Sg^2(I^2(A_{n-3})) \op \Sg
    I^2(B_{n-2}) \\
  d^X(\bt) &= d\bt & \in & I(B_{n-1}) \\
  d^X(\Sg\bt') &= \bt' - \Sg d\bt'
    & \in & I(B_{n-1}) \op I^2(B_{n-2}). 
 \end{align*}
 We also define maps $A_*\xra{i}X_*\xra{p}B_*$ by $i(a)=a$ and 
 \[ p(a + \Sg\al' + \Sg^2\al'' + \bt + \Sg\bt') = f(a) + \tht(b).
 \]
 It is clear that $i$ is injective, the cokernel of $i$ is free in
 each degree, $p$ is surjective, and $pi=f$.
\end{construction}

\begin{proposition}\label{prop-cof-afb-i}
 $(X_*,d^X)$ is a chain complex, and $i$ and $p$ are chain maps.
\end{proposition}
\begin{proof}
 This is a straightforward check of definitions.  We have
 \begin{align*}
  d^Xd^X(a) &= d^2(a) = 0 \\
  d^Xd^X(\Sg\al') &= d^X(\tht(\al')-\Sg d\al' - f_*(\al')) \\
   &= d\tht(\al') - (\tht(d\al')-\Sg d^2\al' - f_*(d\al'))
       - df_*(\al') = 0 \\
  d^Xd^X(\Sg^2\al'') &= d^X(\Sg\al'' + \Sg^2 d\al'' + \Sg f_*(\al'')) \\
   &= (\tht(\al'') - \Sg d\al'' - f_*(\al'')) + 
      (\Sg d\al'' + \Sg^2 d^2\al'' + \Sg f_*(d\al'')) +
      (f_*(\al'') - \Sg d f_*(\al'')) \\
   &= \tht(\al'') = 0 \\
  d^Xd^X(\bt) &= d^2\bt = 0 \\
  d^Xd^X(\Sg\bt') &= d^X(\bt' - \Sg d\bt') \\
   &= d\bt' - (d\bt' - \Sg d^2\bt') = 0.
 \end{align*}
 This shows that $X_*$ is a chain complex, and it is immediate that
 $i$ is a chain map.  We also have
 \begin{align*}
  p(d^X(a)) &= f(da) = df(a) \\
  p(d^X(\Sg\al')) &= p(\tht(\al')-\Sg d\al' - f_*(\al')) 
   = f(\tht(\al'))-\tht(f_*(\al')) = 0 \\
  p(d^X(\Sg^2\al'')) &=
   p(\Sg\al'' + \Sg^2 d\al'' + \Sg f_*(\al'')) = 0 \\
  p(d^X(\bt)) &= p(d(\bt)) = \tht(d(\bt)) = d\tht(\bt) \\
  p(d^X(\Sg\bt')) &= p(\bt'-\Sg d\bt') = \tht(\bt') = 0.
 \end{align*}
 This is easily seen to agree with $dp$, so $p$ is also a chain map.
\end{proof}

\begin{proposition}\label{prop-cof-afb-ii}
 The map $p\:X_*\to B_*$ is a quasiisomorphism (and thus an acyclic
 fibration).
\end{proposition}
\begin{proof}
 Put
 \begin{align*}
  Z_* &= I^2(B_*) \op \Sg I^2(B_*) \\
  Y_* &= \Sg I^2(A_*) \op \Sg^2I^2(A_*) \op Z_* \\
  K_* &= \ker(p\: X_* \to B_*).
 \end{align*}
 Thus, for an element
 $x=a + \Sg\al' + \Sg^2\al'' + \bt + \Sg\bt'\in X_n$ we have
 \begin{align*}
  x \in Z_n &\iff a=0,\;\al'=0,\;\al''=0,\;\tht\bt=0 \\
  x \in Y_n &\iff a=0,\;\tht\al'=0,\;\tht\bt=0 \\
  x \in K_n &\iff fa+\tht\bt=0.
 \end{align*}
 Note that $Z_*\leq Y_*\leq K_*$, and these are inclusions of
 subcomplexes.  As $p$ is surjective, it will suffice to prove that
 $K_*$ is acyclic.  The complexes $Z_*$ and $Y_*/Z_*$ are easily seen
 to be contractible, so it will suffice to show that the complex
 $\ov{K}_*=K_*/Y_*$ is acyclic.  Now put 
 \begin{align*}
  \ov{X}_n &= A_n \op \Sg A_{n-1} \op B_n \\
  d^{\ov{X}}(a + \Sg a' + b) &= (da + a') - \Sg da' + (db - f(a')) \\
  \ov{p}(a + \Sg a' + b) &= f(a) + b.
 \end{align*}
 Using the map
 \[ (a + \Sg\al' + \Sg^2\al'' + \bt + \Sg\bt') \mapsto 
    (a + \Sg\tht(\al') + \tht(\bt))
 \]
 we can identify $\ov{X}_*$ with $X_*/Y_*$, and thus $\ov{K}_*$ with
 the kernel of $\ov{p}$.  Let $C_*$ be the cone on $A_*$, so
 $C_*=A_*\op\Sg A_*$ with $d^C(a+\Sg a')=da + a' - \Sg da$.  Define
 $j\:C_*\to\ov{X}_*$ by $j(a+\Sg a')=a+\Sg a'-f(a)$.  This is easily
 seen to be a chain map and a kernel for $\ov{p}$, so $\ov{K}_*$ is
 isomorphic to $C_*$ and is contractible.
\end{proof}

\begin{definition}
 In the case $A_*=0$ we write $\Gm(B)_*$ for $X_*$.  Thus
 $\Gm(B)_n=I(B_n)\op I^2(B_{n-1})$ with
 $d(\bt+\Sg\bt')=d\bt+\bt'-\Sg d\bt'$, and $\Gm(B)_*$ is free in each
 degree, and we have a surjective quasiisomorphism
 $p\:\Gm(B)_*\to B_*$ given by $p(\bt+\Sg\bt')=\tht(\bt)$.
\end{definition}

\begin{remark}
 Another way to think about $X_*$ is as follows.  Write $j$ for the
 inclusion $I^2\to I$.  We have maps of chain complexes
 \[ I^2(A_*) \xra{\bsm j\\ f_*\esm} 
    I(A_*)\op I^2(B_*) 
    \xra{\bsm \tht & 0 \\ -f_* & j \esm}
    A_*\op I(B_*).
 \]
 The composite of these maps is zero, so we can regard the above
 diagram as a double complex.  The totalisation of this double complex
 is $X_*$.
\end{remark}

We now start to discuss lifting properties.  It will be convenient to
introduce some test objects:
\begin{definition}
 Let $M$ be an abelian group.  We write $\Sg^nM$ for the complex
 consisting of a copy of $M$ in degree $n$.  We also write $C^nM$ for
 the complex consisting of two copies of $M$ in degrees $n$ and $n+1$,
 with the differential between them being the identity.  
\end{definition}

\begin{remark}\label{rem-CS-hom}
 There are easy natural isomorphisms
 \begin{align*}
  \Ch(A_*,C^nM) & \simeq \Hom(A_n,M) &
  \Ch(C^nM,A)   & \simeq \Hom(M,A_{n+1}) \\
  \Ch(A_*,\Sg^nM) & \simeq \Hom(A_n/dA_{n+1},M) &
  \Ch(\Sg^nM,A_*) & \simeq \Hom(M,\ker(d\:A_n\to A_{n-1})).
 \end{align*}
 There is a short exact sequence $\Sg^nM\to C^nM\to\Sg^{n+1}M$;
 applying $\Ch(A_*,-)$ to this gives the left exact sequence
 \[ \Hom(A_n/dA_{n+1},M) \xra{\pi^*} 
    \Hom(A_n,M) \xra{d^*} \Hom(A_{n+1}/dA_{n+2},M).
 \]
\end{remark}

\begin{lemma}\label{lem-free-complex}
 Let $A_*$ be a chain complex that is free in each degree.  Then there
 exists a splitting $A_*=Y_*\op Z_*$ of graded groups and 
 injective maps $d'\:Y_n\to Z_{n-1}$ such that the differential is
 given by $d(y+z)=d'(y)$.  Moreover, $A_*$ is acyclic iff $d'$ is an
 isomorphism iff $A_*$ is contractible.  
\end{lemma}
\begin{proof}
 Put $Z_n=\ker(d\:A_n\to A_{n-1})$ and $B_n=\img(d\:A_{n+1}\to A_n)$,
 so $B_n\leq Z_n\leq A_n$ and therefore $Z_n$ and $B_n$ are free.  We
 can therefore split the epimorphism $d\:A_n\to B_{n-1}$ and thus
 choose a subgroup $Y_n\leq A_n$ such that $d\:Y_n\to B_{n-1}$ is an
 isomorphism.  If $a\in A_n$ we see that there is a unique
 $y\in Y_n$ with $dy=da$, which means that the element $z=a-y$ lies in
 $Z_n$.  Using this we see that $A_n=Y_n\op Z_n$, and it is clear
 that the differential has the stated form.  Now $H_*A$ is the
 cokernel of $d'$ so $A_*$ is acyclic iff $d'$ is an isomorphism, in
 which case $A_*$ is isomorphic to the cone on $Z_*$ and is
 contractible. 
\end{proof}

\begin{proposition}\label{prop-cofibrant}
 The functor $\Ch(A_*,-)$ converts acyclic fibrations to epimorphisms
 iff $A_*$ is free in each degree.
\end{proposition}
\begin{proof}
 Suppose that $\Ch(A_*,-)$ converts acyclic fibrations to
 epimorphisms.  For any surjection $f\:M\to N$ of abelian groups we
 have an acyclic fibration $C^n(f)\:C^nM\to C^nN$, so the map
 \[ \Hom(A_n,M) = \Ch(A_*,C^nM) \xra{f_*}
     \Ch(A_*,C^nN) = \Hom(A_n,M)
 \]
 is surjective.  This means that $A_n$ is projective and thus free.

 Conversely, suppose that $A_n$ is free for all $n$, so we can split
 $A_*$ as $Y_*\op Z_*$ as in Lemma~\ref{lem-free-complex}.  We
 first claim that if $K_*$ is an acyclic complex and $k\:A_*\to K_*$
 is a chain map then $k$ is nullhomotopic.  To see this, put
 $ZK_n=\ker(K_n\xra{d}K_{n-1})$; as $K_*$ is acyclic, the map
 $d\:K_{n+1}\to ZK_n$ is surjective.  Now $k$ gives a map
 $Z_n\to ZK_n$ and $Z_n$ is free so we can choose a lift
 $t\:Z_n\to K_{n+1}$ with $dt=k$.  We now have a homomorphism
 $Y_n\to ZK_n$ given by $y\mapsto k(y)-t(d'y)$, so we can choose a
 lift $s\:Y_n\to K_{n+1}$ with $ds(y)=k(y)-t(d'y)$.  We then put
 $r(y+z)=s(y)+t(z)$ and observe that $dr+rd=k$ as required.

 Now consider an acyclic fibration $q\:L_*\to M_*$, so $q$ is
 surjective and the kernel $K_*$ is acyclic.  Let $j\:K_*\to L_*$ be
 the inclusion.  Suppose we have a chain map $g\:A_*\to M_*$.  As
 $A_*$ is degreewise free and $q$ is surjective we can choose a map
 $h'\:A_*\to L_*$ of graded groups with $qh'=g$.  For $a\in A_n$ put
 $k(a)=\Sg(dh'(a)-h'(da))\in(\Sg K)_n$.  We have
 $dk(a)=\Sg dh'd(a)=k(da)$, so $k$ is a chain map $A_*\to\Sg K_*$.
 By the previous paragraph we can choose maps
 $r_n\:A_n\to(\Sg K)_{n+1}$ with $dr+rd=k$.  We then have
 $r(a)=\Sg r'(a)$ say, and $-dr'+r'd=\Sg^{-1}k$.  It follows that the
 map $h=h'+r'\:A_*\to L_*$ is a chain map with $qh=g$, as required.
\end{proof}

\begin{corollary}\label{cor-split-ses}
 Let $A_*\xra{i}B_*\xra{p}C_*$ be a short exact sequence of chain
 complexes in which $A_*$ is acyclic, and $C_*$ is free in each
 degree.  Then the sequence is split (in the category of chain
 complexes, not just in the underlying category of graded abelian
 groups). 
\end{corollary}
\begin{proof}
 Exactness means that $p$ is surjective, and the kernel is $A_*$ which
 is acyclic by assumption, so $p$ is an acyclic fibration.  The
 Proposition tells us that the induced map
 $\Ch(C_*,B_*)\to\Ch(C_*,C_*)$ is surjective.  In particular,
 $1$ is in the image, so we can choose a chain map $s\:C_*\to B_*$
 with $ps=1$, giving the required splitting.
\end{proof}

\begin{proposition}\label{prop-acyclic-cofibrant}
 The functor $\Ch(A_*,-)$ converts all fibrations to epimorphisms iff
 $A_*$ is contractible and free in each degree.
\end{proposition}
\begin{proof}
 First suppose that $\Ch(A_*,-)$ sends fibrations to epimorphisms.  By
 the previous proposition, $A_*$ is free in each degree.  The evident
 projection $C^nM\to\Sg^{n+1}M$ is a fibration, so the induced map
 \[ \Hom(A_n,M) = \Ch(A,C^nM) \to \Ch(A,\Sg^{n+1}M) =
     \Hom(A_{n+1}/dA_{n+2},M)
 \]
 is surjective.  This is just the map induced by $d$, so we see that
 the map $d\:A_{n+1}/dA_{n+2}\to A_n$ must be a split monomorphism.
 As $A_n$ is free it follows that $A_{n+1}/dA_{n+2}$ is free so the
 projection $A_{n+1}\to A_{n+1}/dA_{n+2}$ must split.  We can
 therefore find $Y_{n+1}\leq A_{n+1}$ such that
 $A_{n+1}=dA_{n+2}\oplus Y_{n+1}$ and $d\:Y_{n+1}\to A_n$ is a split
 monomorphism.  This means that $\ker(d\:A_{n+1}\to A_n)=dA_{n+2}$.
 After doing this for all $n$ we get a splitting $A_*=Y_*\oplus dY_*$
 such that $d\:Y_*\to dY_*$ is an isomorphism.  In other words, $A_*$
 is contractible.
 
 Conversely, suppose we start with the assumption that $A_*$ is
 degreewise free and contractible.  We then have
 $A_*=Y_*\op Z_*$, with the differential given by an isomorphism
 $Y_n\to Z_{n-1}$.  This gives an isomorphism
 $\Ch(A_*,L_*)=\Ab_*(Y_*,L_*)$ and $Y_*$ is projective in $\Ab_*$ so
 this functor preserves epimorphisms, as required. 
\end{proof}

\begin{proposition}\label{prop-lift-abelian}
 Let $\CA$ be an abelian category, and let $A\xra{i}B\xra{p}C$ and
 $K\xra{j}L\xra{q}M$ be short exact sequences.  For any diagram as
 shown,
 \[ \xymatrix{
  {A}
   \ar@{ >->}[d]_{i} 
   \rto^{f} &
  {L}
   \ar@{->>}[d]^{q} \\
  {B}
   \rto_{g} &
  {M}
  }
 \]
 we let $H(f,g)$ denote the set of maps $h\:B\xra{}L$ such that $qh=g$
 and $hi=f$.  Then there is a naturally defined extension
 $K\xra{}T(f,g)\xra{}C$ such that splittings of $T(f,g)$ biject with
 $H(f,g)$.  In particular, if $\Ext^1_{\CA}(C,K)=0$, then $H(f,g)$ is
 always nonempty, so $i$ has the left lifting property with respect to
 $q$. 
\end{proposition}
\begin{proof}
 Consider the following diagram:
 \[ \xymatrix{
  {0}
   \rto
   \dto &
  {A}
   \ar@{=}[r]
   \ar@{ >->}[d]^{\bsm i\\ f\esm} &
  {A}
   \ar@{ >->}[d]^{i} \\
  {L}
   \ar@{ >->}[r]^{\bsm 0\\ 1\esm}
   \ar@{->>}[d]_{q} &
  {B\op L}
   \ar@{->>}[r]^{\bsm 1 & 0\esm} 
   \ar@{->>}[d]^{\bsm -g & q\esm} &
  {B}
   \dto \\
  {M}
   \ar@{=}[r] &
  {M}
   \rto &
  {0}
  }
 \]
 Each column can be regarded as a complex, so the whole diagram is a
 short exact sequence of complexes, leading to a long exact sequence
 of homology groups.  This long exact sequence has only three nonzero
 terms, so it gives a short exact sequence of the form
 $K\xra{k}T\xra{r}C$, where $T=T(f,g)$ is the unique homology group of
 the middle column.  

 If $h\in H(f,g)$, then consider the following diagram:
 \[ \xymatrix{
  {0}
   \dto &
  {A}
   \lto
   \ar@{ >->}[d]^{\bsm i\\ f\esm} &
  {A}
   \ar@{=}[l]
   \ar@{ >->}[d]^{i} \\
  {L}
   \ar@{->>}[d]_{q} &
  {B\op L}
   \ar@{->>}[l]_{\bsm -h & 1\esm} 
   \ar@{->>}[d]^{\bsm -g & q\esm} &
  {B}
   \ar@{ >->}[l]_{\bsm 1\\ h\esm}
   \dto \\
  {M} &
  {M}
   \ar@{=}[l] &
  {0}
   \lto
  }
 \]
 The columns are the same complexes as before, and the horizontal maps
 give a splitting of our previous short exact sequence of complexes,
 and so induce a splitting of the homology group $T$.  

 For the opposite correspondence, we need more information about $T$.
 Let $Z$ be the corresponding group of cycles, which is the kernel of
 the map $(-g,q)\:B\op L\xra{}M$, or in other words, the pullback of
 $q$ and $g$.  We name the maps in the pullback square as follows:
 \[ \xymatrix{
  {Z}
   \rto^{\tg}
   \ar@{->>}[d]_{\tq} &
  {L}
   \ar@{->>}[d]^{q}  \\
  {B}
   \rto_{f} &
  {M}
  }
 \]
 Thus $\tg$ and $\tq$ are just the projections $B\op L\xra{}L$ and
 $B\op L\xra{}B$, restricted to $Z$.

 The differential in our complex is the map $\ti:=(i,f)\:A\xra{}Z$, so
 $T$ is by definition the cokernel of $\ti$; we write $\tp\:Z\xra{}T$
 for the quotient map.  We write $\tk:=(0,j)\:K\xra{}Z$.  One checks
 that the following diagram 
 commutes:
 \[ \xymatrix{
  & {K}
   \ar@{=}[r]
   \ar@{ >->}[d]_{\tk} &
  {K}
   \ar@{ >->}[d]^{k} \\
  {A}
   \ar@{=}[d]
   \ar@{ >->}[r]^{\ti} &
  {Z}
   \ar@{->>}[r]^{\tp}
   \ar@{->>}[d]_{\tq} &
  {T}
   \ar@{->>}[d]^{r}  \\
  {A}
   \ar@{ >->}[r]_{i} &
  {B}
   \ar@{->>}[r]_{p} &
  {C}
  }
 \]
 We also see (by inspection of definitions and diagram chasing) that
 all rows and columns are exact, and that the bottom right square is a
 pullback. 

 Now suppose we are given a splitting of the sequence
 $K\xra{k}T\xra{r}C$, given by a map $n\:C\xra{}T$ with $rn=1$.  By
 the pullback property, there is a unique map $\tn\:B\xra{}Z$ with
 $\tq\tn=1_B$ and $\tp\tn=np\:B\xra{}T$.  We claim that the map
 $h:=\tg\tn\:B\xra{}L$ lies in $H(f,g)$.  Indeed, we first have
 $qh=q\tg\tn=g\tq\tn=g$.  Next, one checks that $\tp(\ti-\tn i)=0$ and 
 $\tq(\ti-\tn i)=0$ so the pullback property tells us that
 $\tn i=\ti\:A\xra{}Z$.  This gives $hi=\tg\ti=f$ as required.

 We leave it to the reader to check that the above constructions are
 mutually inverse.
\end{proof}

\begin{proposition}\label{prop-lift-i}
 A map $i\:A_*\to B_*$ has the left lifting property with respect to
 acyclic fibrations iff $i$ is a cofibration.
\end{proposition}
\begin{proof}
 First suppose that $i$ has the lifting property.  We apply this to
 the acyclic fibration $C^nM\to 0$, remembering that
 $\Ch(A_*,C^nM)=\Hom(A_n,M)$.  We see that for every map $u\:A_n\to M$
 there exists $v\:B_n\to M$ with $vi=u$.  By taking $u$ to be the
 identity map, we see that $i$ is a degreewise split monomorphism,
 with cokernel $C_*$ say.  The map $0\to C_*$ is a pushout of $i$, so
 it has left lifting for acyclic fibrations, which means precisely
 that $\Ch(C_*,-)$ sends acyclic fibrations to epimorphisms.  Thus,
 Proposition~\ref{prop-cofibrant} tells us that $C_*$ is degreewise
 free, so $i$ is a cofibration.

 Conversely, suppose that $i$ is a cofibration, with cokernel $C_*$
 say.  Let $q\:L_*\to M_*$ be an acyclic fibration, so $q$ is
 surjective and the kernel $K_*$ is acyclic.  By
 Proposition~\ref{prop-lift-abelian}, it will suffice to show that
 every short exact sequence $K_*\to T_*\to C_*$ is split, but this
 follows directly from Corollary~\ref{cor-split-ses}.
\end{proof}

\begin{proposition}\label{prop-lift-ii}
 A map $i\:A_*\to B_*$ has the left lifting property with respect to
 all fibrations iff $i$ is an acyclic cofibration.
\end{proposition}
\begin{proof}
 First suppose that $i$ has the lifting property.  The previous result
 tells us that $i$ is a cofibration, with cokernel $C_*$ say.  The map
 $0\to C_*$ is a pushout of $i$ and so again has the lifting property,
 which means (by Proposition~\ref{prop-acyclic-cofibrant}) that $C_*$
 is contractible, so $i$ is an acyclic cofibration.

 Conversely, suppose that $i$ is an acyclic cofibration, and again
 write $C_*$ for the cokernel.  We see from
 Proposition~\ref{prop-acyclic-cofibrant} that every short exact
 sequence $K_*\to T_*\to C_*$ is split, and it follows using
 Proposition~\ref{prop-lift-abelian} that $i$ has left lifting for all
 fibrations. 
\end{proof}

\begin{corollary}
 $\Ch$ is a model category.
\end{corollary}
\begin{proof}
 It is clear that $\Ch$ has finite limits and colimits, that the
 classes of fibrations, cofibrations and weak equivalences are closed
 under retracts, and that the weak equivalences satisfy the
 two-out-of-three condition.  Functorial factorisations are provided
 by constructions~\ref{cons-acf-fib} and~\ref{cons-cof-afb}.   The
 lifting axioms are satisfied by Proposition~\ref{prop-lift-i}
 and~\ref{prop-lift-ii}. 
\end{proof}

\begin{definition}\label{defn-IJ}
 Let $i_n$ denote the map $0\to C^n\Z$, and let $j_n$ denote the
 evident inclusion $\Sg^n\Z\to C^n\Z$.  Put $\CI=\{i_n\st n\in\Z\}$
 and $\CJ=\{j_n\st n\in\Z\}$.
\end{definition}

\begin{proposition}\label{prop-cg}
 A morphism $f\:A\to B$ in $\Ch$ is a fibration iff it has the right
 lifting property with respect to $\CI$.  Similarly, $f$ is an acyclic
 fibration iff it has the right lifting property with respect to
 $\CI\cup\CJ$.  This, the model structure on $\Ch$ is cofibrantly
 generated (as in~\cite[Section 2.1.3]{ho:mc}).
\end{proposition}
\begin{proof}
 From Remark~\ref{rem-CS-hom} we get $\Ch(C^n\Z,A)=A_{n+1}$.  It
 follows that $f$ has the RLP with respect to $i_n$ iff
 $f_{n+1}\:A_{n+1}\to B_{n+1}$ is surjective, and thus that $f$ has
 the RLP with respect to $\CI$ iff $f$ is surjective iff $f$ is a
 fibration.

 Next, it is clear that the maps $i_n$ and $j_n$ are cofibrations, so
 that every acyclic fibration has the RLP with respect to
 $\CI\cup\CJ$.

 Conversely, suppose that $f$ has the RLP with respect to
 $\CI\cup\CJ$.  By the first part of this proof, we see that $f$ is
 surjective, with kernel $K$ say.  From Remark~\ref{rem-CS-hom} again
 we see that $\Ch(\Sg^n\Z,A)=\ker(d\:A_n\to A_{n-1})=Z_nA$.  Using
 this, the RLP for $\CJ$ translates as follows: given $a\in Z_nA$ and
 $b'\in B_{n+1}$ with $db'=fa$, there exists $a'\in A_{n+1}$ with
 $da'=a$ and $fa'=b'$.  This easily implies that $f_*\:H_*A\to H_*B$
 is injective.  We claim that it is also surjective.  To see this,
 consider a homology class $[b^*]\in H_{n+1}B$, so $db^*=0$.  As $f$
 is surjective, we can choose $a^*\in A_{n+1}$ with $fa^*=b^*$, and
 then take $a=da^*\in Z_nA$.  We then have $fa=db^*=0=d0$, so we can
 apply the RLP to the pair $(a,0)$ to get an element $a'\in A_{n+1}$
 with $da'=a$ and $fa'=0$.  The element $a^*-a'$ is then a cycle and
 we have $f_*[a^*-a']=[b^*]$ as required.  This shows that $f$ is an
 acyclic fibration.

 It is clear that for $T=0$ or $T=C^n\Z$ or $T=\Sg^n\Z$, the functor
 $\Ch(T,-)$ preserves all filtered colimits, so these objects are
 small in the strongest possible sense.  
\end{proof}

\begin{proposition}\label{prop-proper}
 The pushout of any weak equivalence along any cofibration is still a
 weak equivalence, as is the pullback of any weak equivalence along
 any fibration.  Thus, the model structure is both left and right
 proper.
\end{proposition}
\begin{proof}
 Suppose we have a pushout square as follows, in which $i$ is a
 cofibration and $f$ is a weak equivalence.
 \[ \xymatrix{
   A_* \ar@{ >->}[r]^i \ar[d]_f & B_* \ar[d]^g \\
   C_* \ar@{ >->}[r]_j & D_*
 } \]
 We must show that $g$ is also a weak equivalence.  From the pushout
 property it is standard and easy to check that $j$ is also a
 cofibration and that the induced map $h\:\cok(i)\to\cok(j)$ is an
 isomorphism.  We therefore get a morphism of short exact sequences of
 chain complexes, and thus a morphism of long exact sequences of
 homology groups:
 \[ \xymatrix{
   H_{n+1}\cok(i) \ar[r]^\dl \ar[d]_{h_*}^\simeq &
   H_nA \ar[r]^{i_*} \ar[d]_{f_*}^\simeq &
   H_nB \ar[r] \ar[d]^{g_*} &
   H_n\cok(i) \ar[r]^\dl \ar[d]_\simeq^{h_*} &
   H_{n-1}A \ar[d]_\simeq^{f_*}   \\
   H_{n+1}\cok(j) \ar[r]_\dl &
   H_nC \ar[r]_{j_*} &
   H_nD \ar[r] &
   H_n\cok(j) \ar[r]_\dl &
   H_{n-1}C
 } \]
 The five lemma shows that $g_*$ is an isomorphism, as required.  This
 proves left properness, and the proof for right properness is
 similar. 
\end{proof}

\begin{proposition}\label{prop-monoidal}
 If $i\:A_*\to B_*$ and $j\:C_*\to D_*$ are cofibrations, and $P_*$ is
 the pushout of
 \[ B_*\ot C_* \xla{i\ot 1} A_*\ot C_* \xra{1\ot j} A_*\ot B_*, \]
 then the induced map $k\:P_*\to B_*\ot D_*$ is again a cofibration.
 Moreover, if at least one of $i$ and $j$ is acyclic, then $k$ is also
 acyclic.  In other words, the pushout product axiom is satisfied, so
 we have a monoidal model structure.
\end{proposition}
\begin{proof}
 Let $p\:B_*\to U_*$ and $q\:D_*\to V_*$ be the cokernels of $i$ and
 $j$ (so $U_*$ and $V_*$ are free in each degree).  We then find that
 the map $p\ot q\:B_*\ot D_*\to U_*\ot V_*$ gives zero when composed
 with $i\ot 1$ or $1\ot j$, so it is also zero when composed with $k$,
 so it induces a chain map $m\:\cok(k)\to U_*\ot V_*$.

 Next, in the category of graded abelian groups, we then have
 splittings $B_*=A_*\oplus U_*$ and $D_*=C_*\oplus V_*$.  From this we
 obtain splittings
 \begin{align*}
  P_* &= (A_*\ot C_*)\oplus(U_*\ot C_*) \oplus (A_*\ot V_*) \\ 
  B_*\ot D_* &= (A_*\ot C_*)\oplus(U_*\ot C_*) \oplus
                (A_*\ot V_*) \oplus (U_*\ot V_*), 
 \end{align*}
 with respect to which $k$ is just the inclusion of the first three
 summands.  This shows that $k$ is injective and that
 $m\:\cok(k)\to U_*\ot V_*$ is an isomorphism.  In particular, the
 cokernel of $k$ is free in each degree, so $k$ is a cofibration.  If
 $i$ is an acyclic cofibration then $U_*$ is contractible (by
 Lemma~\ref{lem-free-complex}) so $U_*\ot V_*$ is also contractible so
 $k$ is also an acyclic cofibration.  The same holds if $j$ is an
 acyclic cofibration.
\end{proof}

\begin{remark}\label{rem-ChR}
 Let $R$ be a ring that is not hereditary (so submodules of free
 modules need not be free).  We then cannot use the methods of this
 note to produce a model structure on the category $\Ch_R$ of chain
 complexes of modules over $R$.  However, we could in principle take
 $R$ as a monoid object in $\Ch$, and regard $\Ch_R$ as the category
 of modules for this monoid.  Then the general framework
 of~\cite[Theorem 4.1]{scsh:amm} would give a model structure on
 $\Ch_R$.  
\end{remark}

\end{document}